\documentclass[12pt]{amsart}

\usepackage{amssymb, amsmath,amsthm}
\usepackage[backref]{hyperref}
\usepackage[alphabetic,backrefs,lite]{amsrefs}
\usepackage{enumerate}
\usepackage{euscript}
\usepackage{color}

\renewcommand{\a}{\alpha}
\renewcommand{\b}{\beta}
\newcommand{\m}{\mathfrak{m}}
\newcommand{\n}{\mathfrak{n}}
\newcommand{\oo}{\otimes}
\newcommand{\om}{\omega}
\newcommand{\p}{\mathfrak{p}}

\newcommand{\bb}[1]{\mathbb{#1}}
\newcommand{\defi}[1]{{\upshape\sffamily #1}}
\newcommand{\ds}{\displaystyle}

\newcommand{\mc}[1]{\mathcal{#1}}
\newcommand{\ol}[1]{\overline{#1}}
\newcommand{\ul}[1]{\underline{#1}}

\newcommand{\Max}{\textrm{Max}}
\newcommand{\Spec}{\textrm{Spec}}
\newcommand{\Tor}{\textrm{Tor}}

\newcommand{\adj}{\textrm{adj}}
\newcommand{\ann}{\textrm{ann}}
\newcommand{\codim}{\textrm{codim}}
\newcommand{\gr}{\textrm{gr}}
\newcommand{\pd}{\textrm{pd}}
\newcommand{\ord}{\textrm{ord}}
\newcommand{\reg}{\textrm{reg}}

\def\lra{\longrightarrow}

\newtheorem{theorem}{Theorem}

\newtheorem{conjecture}[theorem]{Conjecture}

\newtheorem{question}[theorem]{Question}

\theoremstyle{definition}
\newtheorem{definition}[theorem]{Definition}
\newtheorem*{definition*}{Definition}
\newtheorem{example}[theorem]{Example}

\numberwithin{equation}{section}

\author{Craig Huneke}
\address{University of Virginia,  Department of Mathematics, 141 Cabell Drive, Kerchof Hall, Charlottesville, VA 22904}
\email{huneke@virginia.edu}

\author{Claudiu Raicu}
\address{Department of Mathematics, Princeton University, Princeton, NJ 08544\newline
\indent Institute of Mathematics ``Simion Stoilow'' of the Romanian Academy}
\email{craicu@math.princeton.edu}

\thanks{The first author was supported by NSF grant number 1259142.}
\date{July 31, 2013}

\begin{document}

\thispagestyle{plain}

\title{Introduction to Uniformity in Commutative Algebra}

\maketitle

\begin{abstract}
 These notes are based on three lectures given by the first author as part of an introductory workshop at MSRI for the program in Commutative Algebra, 2012-13. The notes follow the talks, but there are extra comments and explanations, as well as a new section on the uniform Artin-Rees theorem. The notes deal with the theme of uniform bounds, both absolute and effective, as well as uniform annihilation of cohomology.
\end{abstract}

\section{Introduction}\label{sec:introduction}

The goal of these notes is to introduce the concept of \defi{uniformity} in Commutative Algebra. Rather than giving a precise definition of what uniformity means, we will try to convey the idea of uniformity through a series of examples. As we'll soon see, uniformity is ubiquitous in Commutative Algebra: it may refer to absolute or effective bounds for certain natural invariants (ideal generators, regularity, projective dimension), or uniform annihilation of (co)homology functors (Tor, Ext, local cohomology). We will try to convince the reader that the simple exercise of thinking from a uniform perspective  almost always leads to significant, interesting,  and fundamental questions and theories.  This theme
has also been discussed in an article by 
 Schoutens \cite{sch}, who shows how uniform bounds can be useful in numerous contexts that we do not consider in
this paper.  

The first section of this paper, based on the first lecture in the workshop, is more elementary and introduces many basic concepts. The next three sections target specific topics and
require more background in general, though an effort has been made to minimize the knowledge needed to read them. Each section has some exercises which the reader might
solve to gain further understanding. The first section in particular has a great many exercises.

We begin to illustrate the theme of uniformity with what is probably the most basic theorem in commutative algebra: 

\begin{theorem}[Hilbert's Basis Theorem {\cite[Thm.~1.2]{eisCA}}]\label{thm:HilbertBasis} If $k$ is a field and $n$ is a non-negative integer, then any ideal in the polynomial ring $S=k[x_1,\cdots,x_n]$ is finitely generated.
\end{theorem}

As it stands, this theorem does give a type of uniformity, namely the property of being finitely generated. But this is quite general, and not absolute or effective.
One first might try for absolute bounds:

\begin{question}\label{que:numgens} Is there an absolute upper bound for the (minimal) number of generators of ideals in $S$?
\end{question}

This has a positive answer in the case $n=1$: $S=k[x_1]$ is a \defi{principal ideal domain}, so any ideal in $S$ can be generated by one element. However, for $n\geq 2$, it is easy to see that such an absolute bound cannot exist: the ideal $I=(x_1,x_2)^N$ can't be generated by fewer than $(N+1)$ elements. One can then try to refine Question~\ref{que:numgens}, which leads us to several interesting variations:

\begin{question}\label{que:numgensrefined} Is there an absolute upper bound for the number of generators of an ideal $I$ in $S$, if
\begin{enumerate}[(a)]
 \item we assume that $I$ is prime?

 \item $I$ is homogeneous and we impose bounds on the degrees of the generators of $I$?
 
 \item we are only interested in the generation of $I$ up to radical? (Recall that the \defi{radical $\sqrt{I}$} is the set $\{f\in S:f^r\in I\textrm{ for some }r\}$.)
\end{enumerate}
\end{question}

For part (a) we have a positive answer in the case $n=2$: any prime ideal $I\subset k[x_1,x_2]$ is either maximal, or has height one, or is zero, so it can be generated by at most two elements (because the ring is a UFD, height one primes are principal; for maximal ideals, see Exercise~\ref{exer:maxngens}). However, for $n\geq 3$ the assumption that $I$ is prime is not sufficient to guarantee an absolute bound for its number of generators: in fact in \cites{moh1,moh2} a sequence of prime ideals $\p_n\subset k[x,y,z]$ is constructed, where the minimal number of generators of $\p_n$ is $n+1$. For part (b), if we assume that $I$ is generated in degree at most $d$, then the absolute bound for the number of generators of $I$ is attained when $I=\m^d$ is a power of the maximal ideal $\m$, and is given by the binomial coefficient $\ds{n+d-1\choose n-1}$ (see Exercise~\ref{exer:maxgenshomogeneous}). Part (c) is already quite subtle: every ideal $I\subset S$ is generated up to radical by $n$ elements \cites{eis-evans,
storch}.

Another variation of Question~\ref{que:numgens} is to ask whether one can find effective lower bounds for the number of generators of an ideal $I$ of $S$. One such bound is obtained as a consequence of Krull's Hauptidealsatz (Principal Ideal Theorem), in terms of the \defi{codimension} of the ideal $I$:
\[\codim(I)=\dim(S)-\dim(S/I)=n-\dim(S/I).\]

\begin{theorem}[{\cite[Ch.~12]{matsumuraCA},\cite[Ch.~10]{eisCA}}]\label{thm:Krull}
 The number of generators of $I$ is at least as large as the codimension of $I$ (and this inequality is sharp).
\end{theorem}

In order to discuss further uniformity statements, we need to expand the set of invariants that we associate to ideals, and more generally to modules over the polynomial ring $S$. We start with the following:

\begin{definition}[Hilbert Function/Series] Let $M=\bigoplus_{i\in\bb{Z}}M_i$ denote a finitely generated graded $S$-module, written as the sum of its homogeneous components (so that $S_i\cdot M_j\subset M_{i+j}$). The \defi{Hilbert function} $h_M:\bb{Z}\to\bb{Z}_{\geq 0}$ is defined by
\[h_M(i)=\dim_k(M_i).\]
We write $M(d)$ for the shifted module having $M(d)_i=M_{d+i}$. It follows that $h_{M(d)}(i)=h_M(i+d)$.

The \defi{Hilbert series $H_M(z)$} is the generating function associated to $h_M$:
\[H_M(z)=\sum_{i\in\bb{Z}} h_M(i)\cdot z^i.\]
\end{definition}

In the case when $M=S$ is the polynomial ring itself, we have 
\begin{equation}\label{eq:hilbFcnS}
\begin{aligned}
h_S(d)=
 \begin{cases}
\ds{n+d-1\choose n-1} & d\geq 0, \\
0 & d<0.
\end{cases}
\end{aligned}
\end{equation}
The Hilbert series of $S$ takes the simple form
\[H_S(z)=\frac{1}{(1-z)^n}.\]
It is a remarkable fact, which we explain next, that the Hilbert series of any finitely generated graded $S$-module is a rational function. An equivalent statement is contained in the following theorem of Hilbert.

\begin{theorem}[{\cite[Ch.~10]{matsumuraCA},\cite[Ch.~12]{eisCA}}]\label{thm:Hilbpoly} If $M$ is a finitely generated graded $S$-module, then there exists a polynomial $p_M(t)$ with rational coefficients, such that $p_M(i)=h_M(i)$ for sufficiently large values of $i$. 
\end{theorem}

The polynomial $p_M$ is called the \defi{Hilbert polynomial} of $M$. Since the theorem is true for $M=S$ (as shown by (\ref{eq:hilbFcnS})), it holds for free modules as well. To prove it in general, it is then enough to show that any $M$ can be approximated by free modules in such a way that its Hilbert function is controlled by the Hilbert functions of the corresponding free modules.\begin{footnote}{Although we are concentrating on the graded case, the Hilbert function of a local ring
can be defined easily by passing to the associated graded ring. A remarkable   uniform result about Hilbert functions was proved by Srinvias and Trivedi \cite{ST}:
if the local ring is Cohen-Macaulay,  and we fix its dimension and multiplicity, then there are only finitely many possible Hilbert functions.}\end{footnote} Such approximations are realized via exact sequences. When working with graded modules, we will assume that every homomorphism $f:M\to N$ has degree $0$, i.e. $f(M_i)\subset N_i$ for all $i\in\bb{Z}$. It follows that any short exact sequence of graded modules
\[0\lra M\lra N\lra K\lra 0\]
restricts in degree $i$ to an exact sequence
\[0\lra M_i\lra N_i\lra K_i\lra 0,\]
yielding $h_N(i)=h_M(i)+h_K(i)$, and therefore
\[H_N(z)=H_M(z)+H_K(z).\]
Now if $f\in S_d$ is a form of degree $d$, then $K=S/(f)$ can be approximated by free modules via the exact sequence
\[0\lra S(-d)\overset{f}{\lra} S\lra K\lra 0.\]
It follows that 
\[h_{K}(i)=h_S(i)-h_S(i-d)={i+n-1\choose n-1} - {i-d+n-1\choose n-1},\]
which is a polynomial for $i\geq d$. For an arbitrary finitely generated graded module $M$, a similar approximation result holds, having Theorem~\ref{thm:Hilbpoly} as a direct consequence:

\begin{theorem}[Hilbert's Syzygy Theorem {\cite[Ch.~18]{matsumuraCA},\cite[Thm.~1.13]{eisCA}}]\label{thm:Hilbsyzygy}
 If $M$ is a finitely generated graded $S$-module, then there exists a finite minimal graded free resolution
\begin{equation}\label{eq:minresM}
0\lra F_n\lra F_{n-1}\lra\cdots\lra F_1\lra F_0\lra M\lra 0.
\end{equation}
\end{theorem}

In this statement, \defi{minimal} just means that the entries of the matrices defining the maps between the free modules in the resolution have entries in the homogeneous maximal ideal $\m$ of $S$. Writing $F_i=\bigoplus_{j\in\bb{Z}}S(-j)^{\b_{i,j}}$, we call the multiplicities $\b_{i,j}=\b_{i,j}(M)$ the \defi{graded Betti numbers} of $M$. We say that $M$ has a \defi{pure resolution} if for each $i$ there is at most one value of $j$ for which $\b_{i,j}\neq 0$. It has a \defi{linear resolution} if $\b_{i,j}=0$ for $i\neq j$ (or more generally if there exists $c\in\bb{Z}$ such that $\b_{i,j}=0$ for $i-j\neq c$). It is not immediately obvious from (\ref{eq:minresM}) that the Betti numbers are uniquely determined by $M$, but this follows from their alternative more functorial characterization \cite[Prop.~1.7]{eis-syzygies}:
\[\b_{i,j}=\dim_k\Tor_i^S(M,k)_j.\]
Here we think of $i$ as the \defi{homological degree}, and of $j$ as the \defi{internal degree}.

\begin{example}[The Koszul complex]\label{ex:Koszulcpx}
 If we take $M$ equal to $k=S/(x_1,\cdots,x_n)$, the residue field, then its minimal graded free resolution is pure (even linear), given by the \defi{Koszul complex} on $x_1,\cdots,x_n$:
\[0\lra S(-n)^{n\choose n}\lra\cdots\lra S(-i)^{n\choose i}\lra\cdots\lra S(-1)^n\lra S\lra k\lra 0.\]
The graded Betti numbers are given by
\[\b_{i,j}(k)=
\begin{cases}
0 & j\neq i, \\	
\ds{n \choose i} & j=i.
\end{cases}
\]
\end{example}

The graded Betti numbers of a module are recorded into the \defi{Betti table}, where the entry in row $i$ and column $j$ is $\b_{i,i+j}$:
\[\begin{array}{c|cccc}
 & 0 & 1 & 2 & \cdots \\ \hline
0 & \b_{0,0} & \b_{1,1} & \b_{2,2} & \cdots \\
1 & \b_{0,1} & \b_{1,2} & \b_{2,3} & \cdots \\
\vdots & \vdots & \vdots & \vdots & \ddots \\
\end{array}
\]
The Koszul complex in Example~\ref{ex:Koszulcpx} has Betti table
\[\begin{array}{c|cccccccc}
 & 0 & 1 & 2 & \cdots & i & \cdots & n-1 & n \\ \hline
0 & 1 & n & \ds{n\choose 2} & \cdots & \ds{n\choose i} & \cdots & n & 1\\
\end{array}
\]

The Hilbert function, the Hilbert series, and therefore the Hilbert polynomial of $M$ can all be read off from the Betti table of $M$. We have
\[h_M(i)=\sum_{l=0}^n(-1)^l\cdot\sum_j{i-j+n-1\choose n-1}\b_{l,j},\]
and
\[H_M(z)=\frac{\sum_j\left(\sum_l (-1)^l\b_{l,j}\right)\cdot z^j}{(1-z)^n}.\]
In addition, the following basic invariants of $M$ are also encoded by the Betti table:
\begin{enumerate}
 \item \emph{Dimension}:
\[\dim(M)=\dim(S/\ann(M))=\deg(p_M)+1.\]

 \item \emph{Multiplicity}:
\[e(M) = (\dim(M)-1)!\cdot(\textrm{leading coefficient of }p_M).\]

 \item \emph{Projective dimension}:
\[\pd(M) = \textrm{length of the Betti table }(\textrm{the index of the last non-zero column}).\]

 \item \emph{Regularity}.
\[\reg(M) = \textrm{width of the Betti table }(\textrm{the index of the last non-zero row}).\]
\end{enumerate}

Some of the most fundamental problems in Commutative Algebra involve understanding the (relative) uniform properties that these invariants exhibit. For example, it is true generally that
\begin{equation}\label{eq:codimpdim}
\codim(I)\leq\pd(S/I)\leq n, 
\end{equation}
where the first inequality is a consequence of the Auslander-Buchsbaum formula \cite[Thm.~19.9]{eisCA}, while the second follows from Theorem~\ref{thm:Hilbsyzygy}. If the equality $\codim(I)=\pd(S/I)$ holds in (\ref{eq:codimpdim}) then we say that $S/I$ is \defi{Cohen-Macaulay}.

How can one best use the idea of uniformity when considering the Hilbert syzygy theorem? There
are several ways. One can relax the condition that $S$ be a polynomial ring, and consider possibly
infinite resolutions. But then what type of questions should be asked? We will consider this idea
below. One could also ask for absolute bounds on the invariants (1)--(4) under certain restrictions on the degrees of the generators of $I$. One such question is due to Mike Stillman:

\begin{question}[{\cite[Problem~3.8.1]{ps-open}}] Fix positive integers $d_1,\cdots,d_m$. Is $\pd(S/I)$ bounded when $S$ is allowed to vary over the polynomial rings in any number of variables, and $I$ over all the ideals $I=(f_1,\cdots,f_m)$, where each $f_i$ is a homogeneous polynomial of degree $d_i$? With the same assumptions, is $\reg(S/I)$ bounded?
\end{question}

It has been shown by Giulio Caviglia that the existence of a uniform bound for $\pd(S/I)$ is equivalent to the existence of one for $\reg(S/I)$ \cite[Thm.~29.5]{peeva}.

There has been an increasing amount of work on Stillman's question. Some of the recent papers
posted on the ArXiv include one by Ananyan and Hochster \cite{AH} which gives a positive answer
to Stillman's question when all the $d_i = 2$. The paper of Huneke, McCullough, Mantero, and
Seceleanu \cite{HMMS}, gives a sharp upper bound if all $d_i = 2$, and the codimension is also $2$.
See the exercises for more references. 

We mentioned the idea of passing to infinite resolutions. What could possibly be an analogue of
Hilbert's theorem in this case? One such analogue was proved by Hsin-Ju Wang \cite{W1,W2}. His
theorem gives an effective Hilbert syzygy theorem in the following sense. If $R$ is
regular of dimension $d$, then the syzygy theorem implies that modules have projective dimension
at most $d$, and in particular for every pair of $R$-modules $M$ and $N$, $Ext^{d+1}_R(M,N) = 0.$  In particular, if $R_\p$ is regular for some prime ideal $\p$,
then there is an element not in $\p$ which annihilates any given one of these higher Ext modules.  If the singular locus is
closed, this means that there is a power of the ideal defining it which annihilates any fixed $Ext^{d+1}_R(M,N)$. What one can
now ask is a natural question from the point of view of uniformity: is there a {\it uniform} annihilator
of these Ext modules as $M$ and $N$ vary? The result of Hsin-Ju Wang answers this question affirmatively:

\begin{theorem}  Let $(R,\m)$ be a $d$-dimensional complete Noetherian local ring, and let
$I$ be the ideal defining the singular locus of $R$. Let $M$ be a finitely generated $R$-module.
There exists an integer $k$ such that for all $R$-modules $N$, 
$$I^kExt^{d+1}_R(M,N) = 0.$$
\end{theorem}

We can ask for more: can the uniform annihilators be determined effectively? The answer is yes.
If $R = S/\p$
is a quotient of a polynomial ring by a prime ideal, it is well-known (see \cite[Sec.~16.6]{eisCA}) that the singular locus is
inside the closed set determined by the appropriate size minors of the Jacobian matrix, and is equal
to this closed set if the ground field is perfect. Wang proves that one can use elements in the Jacobian ideal
to annihilate these Ext modules, so his result is also {\it effective} in the sense
that specific elements of the ideal $I$ can be constructed from a presentation of the algebra. 

Before we leave this first section, let's look at yet another famous theorem of Hilbert, his Nullstellensatz, which
identifies the nilradical of an ideal $I$ in a polynomial ring $S$ as the intersection of all maximal ideals that contain $I$.
What can one ask in order to change this basic result into a uniform statement?  One answer is that as it stands, the Nullstellensatz
is a theoretical description of the nilradical of $I$, but it is not effective in the sense that information about $I$ is not tied to
information about its nilradical.  There has been considerable work on making the Nullstellensatz ``effective". We quote one
result of Koll\'ar's \cite{Ko} (somewhat simplified):

\begin{theorem}  Let $I$ be a homogeneous ideal in $S = k[x_1,...,x_n]$.  Write $I = (f_1,...,f_m)$, where each $f_i$ is
homogeneous of degree $d_i\geq 3$.  Let $q$ be the minimum of $m$ and $n$. If we let $D = d_1d_2\cdots d_q$, then
$$(\sqrt{I})^D\subseteq I.$$
\end{theorem}

There are many variations, which include quadrics and non-homogeneous versions.  In fact Koll\'ar's theorem itself is more detailed
and specific. Notice that the number of the degrees $d_i$ in the product is absolutely bounded by the number of variables, even if the
number of generators of $I$ is extremely large compared to the dimension. The version we give is sharp, however, as the next example shows.

\begin{example} Let $S = k[x_1,...,x_n]$ be a polynomial ring
over a field $k$. Fix a degree $d$. Set \[f_1 = x_1^d,\ f_2 = x_1x_n^{d-1}-x_2^d,...,\ f_{n-1} = x_1x_n^{d-1}-x_{n-1}^d.\] If $I$ is
the ideal generated by these forms of degree $d$, then it is easy to see that the nilradical of $I$ is the ideal $(x_1,...,x_{n-1})$.
Moreover, $x_{n-1}^D\in I$ for $D = d^{n-1}$, but not for smaller values. We leave this fact to the reader to check.
\end{example}

\subsection*{Exercises:} \ 

\bigskip

\begin{enumerate}
\item[]\textbf{Generators.}
\medskip

\item\label{exer:maxngens} Let $k$ be a field. Prove that every maximal ideal in $k[x_1,\cdots,x_n]$ is generated by
$n$ elements. In particular, every prime ideal in $k[x,y]$ is generated by at most two 
elements.

\item\label{exer:maxgenshomogeneous} Suppose that $I$ is a homogeneous ideal in $S = k[x_1,\cdots,x_n]$ generated by forms of
degrees at most $d$, such that every variable is in the radical of $I$.
 Prove that $I$ can be generated by at most the number of minimal
generators of $\m^d$, where $\m = (x_1,\cdots,x_n)$.
Is the same statement true if one doesn't assume that the radical of $I$ contains $\m$?

\item Let $R$ be a standard graded ring over an infinite field, with homogeneous maximal ideal $\m$. We say that an $\m$-primary
homogeneous ideal $I$ is $\m$-full if for every general linear form $\ell$, $\m I:\ell = I$.
Prove that if $I$ is $\m$-full and $J$ is homogeneous and contains $I$, then the minimal
number of generators of $J$ is at most the minimal number of generators of $I$. 

\item Let $\p$ be a homogeneous prime ideal of a polynomial ring $S$ such that $\p$ contains no linear forms.
It is not known whether or not $\p$ is always generated by forms of degrees at most the
multiplicity of $S/\p$. Can you find examples where this estimate is sharp? What about if
$\p$ is not prime?

\medskip
\item[]\textbf{Radicals.}
\medskip

\item Let $M$ be an $n$ by $n$ matrix of indeterminates over the complex numbers $\mathbb C$, and let
$I$ be the ideal generated by the entries of the matrix $M^n$. Find $n$ polynomials generating an ideal with the same radical
as that of $I$. (Hint: use linear algebra.)

\item It is an unsolved problem whether or not every non-maximal prime ideal in a polynomial ring
$S = k[x_1,\cdots,x_n]$ can be generated up to radical by $n-1$ polynomials. Here is an explicit
example in which the answer is not known, from Moh. Let $\p$ be the defining ideal of the curve $k[t^6+t^{31}, t^8,t^{10}]$
in a polynomial ring in $3$ variables. Can you find a set of generators of $\p$?
It is conjectured that $\p$ is generated up to radical by $2$ polynomials. Why is this the least
possible number of polynomials that could generate $\p$ up to radical?
If the characteristic of $k$ is positive it is known that $\p$ is generated up to radical by $2$
polynomials. Assuming the characteristic is equal to $2$, find such polynomials.

\item Let $k=\bb{C}$ be the field of complex numbers, and let $\p$ be the defining ideal of the surface
$k[t^4,t^3s,ts^3,s^4]$. Find generators for $\p$, and find three polynomials which generate
$\p$ up to radical. It is unknown whether or not there are $2$ polynomials which generate
$\p$ up to radical, although this is known in positive characteristic. 

\item Let $S$ be a polynomial ring, and let $I$ be generated by forms of degrees $d_1,\cdots,d_s$.
Suppose that $f$ is in the radical of $I$, so that there is some $N$ such that
$f^N\in I$. Is there an effective bound for $N$? Take a guess. Find the best example you can
to see that $N$ must be large.

\item
Let $R$ be a regular local ring, and let $I \subseteq R$ be an ideal such that
$R/I$ is Cohen-Macaulay. Let 
\[F_\bullet := 0 \to F_n \overset{f_n}{\to}\dots \overset{f_1}{\to} R \to R/I \to 0\]
be a minimal free resolution of $R/I$.  Show that
$$
\sqrt I = \sqrt{I(f_1)} = \dots = \sqrt{I(f_n)},
$$
where $I(f_i)$ is the ideal of $R$ generated by the $k_i$--minors of $f_i$, where $k_i$ is maximal with the property that the $k_i$--minors of $f_i$ are not all zero.

\medskip

\item[]\textbf{Stillman's Question.}

\medskip

\item Let $S$ be a polynomial ring and let $I$ be an ideal generated by two forms. Show that the projective
dimension of $S/I$ is at most $2$. What well-known statement is this equivalent to?

\item Let $S$ be a polynomial ring. It is known that if $I$ is generated by three quadrics, then
the projective dimension of $S/I$ is at most $4$. Find an example to see that $4$ is attained,
and try to prove this statement.

\item The largest known projective dimension of a quotient $S/I$ where $I$ is generated by
three cubics and $S$ is a polynomial ring is $5$. Can you find such an example? (Hard.)

\item Prove the following strong form of Stillman's problem for monomial ideals: if
$I$ is generated by $s$ monomials in a polynomial ring $S$, then the projective dimension
of $S/I$ is at most $s$. 

\item What about binomial ideals? Is there a bound similar to that in the previous question?

\medskip

\item[]\textbf{Infinite Resolutions.}

\medskip

\item If $R = S/I$ where $S$ is a polynomial ring and $I$ has a Gr\"obner basis of quadrics, then
$R$ is Koszul, i.e. the residue field has a linear resolution. Prove this.

\item Suppose that $R = S/I$, where $S$ is a polynomial ring, and $I$ is homogeneous. If the regularity
of the residue field of $R$ is bounded, show that the regularity of every finitely generated
graded $R$-module $M$ is also bounded.

\item Find an example of a resolution of the residue field of a standard graded algebra so that the degrees of the
entries of the matrices in a minimal resolution (after choosing bases for the free modules) are at least any
fixed number $N$. It is a conjecture of Eisenbud-Reeves-Totaro that one can always choose bases of
the free modules in the resolution of a finitely generated graded module so that the entries in
the whole (usually infinite) set of matrices are bounded.

\item Let $R$ be a Cohen-Macaulay standard graded algebra which is a domain of multiplicity $e$. Prove that the $i$-th total
Betti number (the sum of all $\b_{i,j}$ for $j\in\bb{Z}$) of any quotient $R/I$ is at most $e$ times the $(i-1)$-st total Betti number of $R/I$ for large $i$. 
What sort of uniformity for total Betti numbers might one hope for?

\medskip

\item[]\textbf{Relations between Invariants.}

\medskip

\item Try to imagine a conjecture about effective bounds relating the multiplicity of $S/\p$, where $S$ is
a polynomial ring and $\p$ is a homogeneous prime not containing a linear form, and the regularity of
$S/\p$. Why should there be any relationship? Try the case in which $\p$ is generated by a regular
sequence of forms. 

\item Let $S$ be a polynomial ring, and let $I$ be an ideal generated by square-free monomials.
The  multiplicity of $S/I$ is just the number of minimal primes $\p$ over $I$ such that
the dimension of $S/\p$ is maximal.  Can you say
anything about the regularity? For example, what if $I$ is the edge ideal of a graph (see the discussion following Question~\ref{que:GVV})?

\item Is there any relationship at all between the projective dimension (resp. regularity)
of a quotient $S/I$ ($S$ a polynomial ring, $I$ a homogeneous ideal) and the projective
dimension (resp. regularity) of $S/\sqrt{I}$?  Try to give examples or formulate a problem.

\item Answer the previous question when $I$ is generated by monomials.
\end{enumerate}

\section{Reduction to characteristic $p$ and integral closure}\label{sec:redmodp}

In this section we will discuss the solution to an uniformity question of John Mather, and illustrate in the process two important concepts in Commutative Algebra: reduction to characteristic $p$, and integral closure. Throughout this section, $S$ will denote the power series ring $\bb{C}[[x_1,\cdots,x_n]]$. Given $f\in S$, we will write $f_i$ for its $i$-th partial derivative $\ds\frac{\partial f}{\partial x_i}$. We write $J(f)$ for the \defi{Jacobian ideal} of $f$, $J(f)=(f_1,\cdots,f_n)$.

\begin{question}[Mather {\cite[Que.~13.0.1]{hun-swa}}]\label{que:Mather}
 Consider an element $f\in S$ satisfying $f(0)=0$. Does there exist an uniform integer $N$ such that $f^N\in J(f)$?
\end{question}

The answer to this turns out to be positive as we'll explain shortly, and in fact one can take $N=n$. Notice however that there is no a priori reason for such an $N$ to even exist, that is, for $f$ to be contained in $\sqrt{J(f)}$. Let's first look at some examples:
\begin{itemize}
 \item If $f=x_1^2+x_2^2$ then $J(f)=(2x_1,2x_2)$, so $f\in J(f)$.
 
 \item If $f=x^2-x$ then since $f_1=2x-1$ is a unit, we get $f\in J(f)$.
 
 \item If $f$ is a \defi{homogeneous polynomial}, or more generally a \defi{quasi-homogeneous} one, then $f\in J(f)$ (recall that $f$ is said to be \defi{quasi-homogeneous} if there exist weights $d$ and $\om_i\in\bb{Z}_{\geq 0}$ with the property that $f(t^{\om_1}x_1,\cdots,t^{\om_n}x_n)=t^d f(x_1,\cdots,x_n)$). The conclusion $f\in J(f)$ follows from the quasi-homogeneous version of Euler's formula:
\[\sum_{i=1}^n \om_i\cdot x_i\cdot f_i=d\cdot f.\]

 \item Even when $f$ is not quasi-homogeneous, there might exist an analytic change of coordinates which transforms it into a quasi-homogeneous polynomial, so that the conclusion $f\in J(f)$ still holds. For example, if $f=(x_1-x_2^2)\cdot(x_1-x_2^3)$, one can make the change of variable $y_1=x_1-x_2^2$, $y_2=x_2\cdot\sqrt{1-x_2}$ and get $f=y_1\cdot(y_1+y_2^2)$ which is quasi-homogeneous for $\om_1=2$, $\om_2=1$ and $d=4$. The following theorem gives a partial converse to this observation:
\end{itemize}

\begin{theorem}[Saito {\cite{saito}}]\label{thm:saito}
 If the hypersurface $f=0$ has an isolated singularity (or equivalently $\sqrt{J(f)}=(x_1,\cdots,x_n)$) then $f$ is quasi-homogeneous (with respect to some analytic change of coordinates) if and only if $f\in J(f)$.
\end{theorem}

We begin answering Question~\ref{que:Mather} by looking first at the case $n=1$: writing $t$ for $x_1$, $S=\bb{C}[[t]]$ is a DVR, so any non-zero $f\in S$ can be written as $f=t^i\cdot u$, where $u(0)\neq 0$, i.e. $u$ is a unit; we get
\[\frac{\partial f}{\partial t}=t^{i-1}\cdot\left(i\cdot u+t\cdot\frac{\partial u}{\partial t}\right)=t^{i-1}\cdot\textrm{unit},\]
so $f\in J(f)$. This calculation in fact shows that $f\in J(f)$ even for $n>1$, provided that we first make a ring extension to a power series ring in one variable. More precisely, assume that $n>1$ and consider an embedding $S\hookrightarrow K[[t]]$, given by $x_i\mapsto x_i(t)$, where $K$ is any field extension of $\bb{C}$. The above calculation shows that $f$ is contained in the ideal of $K[[t]]$ generated by the derivative $df/dt$. The Chain Rule yields
\[\frac{df}{dt}=\sum_{i=1}^n f_i\cdot\frac{dx_i}{dt}\in J(f)\cdot K[[t]],\]
so we conclude that $f\in J(f)\cdot K[[t]]$. This motivates the following

\begin{definition}[Integral closure of ideals]\label{def:idealintclosure}
 Given an ideal $I\subset S$, the \defi{integral closure} $\ol{I}$ of $I$ is defined by
\[
\begin{split}
\ol{I}=\{g\in S:\varphi(g)\in\varphi(I)\ &\textrm{ for every field extension }\bb{C}\subset K, \\
&\textrm{ and every }\bb{C}\textrm{--algebra homomorphism }\varphi:S\to K[[t]]\}. 
\end{split}
\]
\end{definition}

We have thus shown that $f\in\ol{J(f)}$, so our next goal is to understand better the relationship between an ideal and its integral closure. We have the following result.

\begin{theorem}[Alternative characterizations of integral closure {\cite[Thm.~6.8.3,\ Cor.~6.8.12]{hun-swa}}]\label{thm:intclosure}
 Given an ideal $I\subset S$ and an element $g\in S$, the following statements are equivalent:
\begin{enumerate}[(a)]
 \item $g\in\ol{I}$.
 
 \item There exist $k\in\bb{Z}_{\geq 0}$ and $s_i\in I^i$ for $i=1,\cdots,k,$ such that
\[g^k = s_1\cdot g^{k-1}+\cdots+s_i\cdot g^{k-i}+\cdots+s_k.\]

 \item There exists $c\in S\setminus\{0\}$ such that 
 \[c\cdot g^m\in I^m\textrm{ for every }m\in\bb{Z}_{\geq 0}.\]
\end{enumerate}
\end{theorem}


It is worth thinking for a moment about the differences between these three characterizations. In fact, they are very different, and
we shall need all three of them. The first characterizes integral closure in a non-constructive way, since the definition depends
on arbitrary homomorphisms to discrete valuation rings. Nonetheless, we have seen the power of this definition by using it to show
that power series are integral over the ideal generated by their partials. The second characterization shows that one needs only a finite
set of data to determine integral closures. In particular, it is clear from this characterization that integral closure behaves well under numerous operations
such as homomorphisms. Finally, the third characterization is the easiest to use in the sense that it is a weak condition, but
the condition by its nature involves  an infinite set of equations. 

Since $f\in\ol{J(f)}$, part (b) of Theorem~\ref{thm:intclosure} implies that $f^k\in J(f)$ for some $k$, but in principle $k$ could depend on $f$. The goal of the rest of this section is to show that we can choose $k=n$, independently of $f$. We will do so by passing to characteristic $p>0$. A word of caution is in order here, which is that the conclusion $f\in\sqrt{J(f)}$ fails in positive characteristic: if $f=g^p$, then $J(f)=0$. Nevertheless, we will prove the following:

\begin{theorem}[{\cites{bri-sko,lip-sat}, \cite[Ch.~13]{hun-swa}}]\label{thm:A}
 Assume that $R$ is a regular local ring of characteristic $p>0$, and consider an ideal $J=(g_1,\cdots,g_t)$ in $R$. If $g\in\ol{J}$ then $g^t\in J$.
\end{theorem}

The advantage of working in positive characteristic is the existence of the \defi{Frobenius endomorphism} $F$ sending every element $x$ to $x^p$. In the case of a regular local ring, the Frobenius endomorphism is in fact flat \cite{kunz}, which yields the following:

\begin{theorem}[Test for Ideal Membership]\label{thm:idealmembership}\begin{footnote}{This test for ideal membership has been conceptualized into an important closure operation called tight closure (see \cite{HH1}). If $R$ is a Noetherian ring of
characteristic $p$, $I$ is an ideal, and $x\in R$, we say that $x$ is in the {\it tight closure}  of $I$ if there exists an element $c\in R$, not in any minimal prime of $R$,
such that $cx^{p^e}\in I^{[p^e]}$ for all large $e$. The set of all elements in the tight closure of $I$ forms a new ideal $I^*$, called the {\it tight closure of $I$}.}\end{footnote}
 Consider a regular local ring $R$ of characteristic $p>0$, an ideal $J=(g_1,\cdots,g_t)$, and an element $g\in R$. We have that $g\in J$ if and only if there exists $c\in R\setminus\{0\}$ such that for all $e$, $c\cdot g^{p^e}$ is in the \defi{Frobenius power} 
\[J^{[p^e]}:=(g_1^{p^e},\cdots,g_t^{p^e}).\]
\end{theorem}

\begin{proof}
 Assume that $g\notin J$ so that $(J:g)=\{x\in R:x\cdot g\in J\}$ is a proper ideal, and consider the exact sequence
\begin{equation}\label{eq:colonses}
0\lra R/(J:g)\overset{g}{\lra} R/J\lra R/(J,g)\lra 0.
\end{equation}
Since $F$ is flat, pulling back (\ref{eq:colonses}) along $F^e$ preserves exactness, yielding the sequence
\begin{equation}\label{eq:F*ses}
0\lra R/(J:g)^{[p^e]}\overset{g^{p^e}}{\lra} R/J^{[p^e]}\lra R/(J,g)^{[p^e]}\lra 0.
\end{equation}
Since $(J,g)^{[p^e]}=(J^{[p^e]},g^{[p^e]})$, we get by comparing (\ref{eq:F*ses}) with the analogue of (\ref{eq:colonses})
\[0\lra R/(J^{[p^e]}:g^{p^e})\overset{g^{p^e}}{\lra} R/J^{[p^e]}\lra R/(J^{[p^e]},g^{p^e})\lra 0\]
that $(J:g)^{[p^e]}=(J^{[p^e]}:g^{p^e})$. The condition $c\cdot g^{p^e}\in J^{[p^e]}$ for all $e$ then becomes
\[c\in\bigcap_{e\geq 0}(J:g)^{[p^e]}\subset\bigcap_{e\geq 0}(J:g)^{p^e}=0,\]
where the last equality follows from the Krull Intersection Theorem \cite[Cor.~5.4]{eisCA}.
\end{proof}

\begin{proof}[Proof of Theorem~\ref{thm:A}] By Theorem~\ref{thm:idealmembership}, it suffices to find $c\neq 0$ such that $c\cdot(g^t)^{p^e}\in J^{[p^e]}$. Since $g\in\ol{J}$, we know by Theorem~\ref{thm:intclosure}(c) that there exists $c\neq 0$ such that $c\cdot g^m\in J^m$ for all $m$. Taking $m=t\cdot p^e$, we have in particular that $c\cdot(g^t)^{p^e}\in J^{t\cdot p^e}$. Since $J^{t\cdot p^e}$ is generated by monomials $g_1^{i_1}\cdots g_t^{i_t}$, with $i_1+\cdots+i_t=t\cdot p^e$, for each such monomial at least one of the exponents $i_j$ satisfies $i_j\geq p^e$. It follows that $J^{t\cdot p^e}\subset J^{[p^e]}$, so $c\cdot(g^t)^{p^e}\in J^{[p^e]}$, concluding the proof of the theorem.
\end{proof}

We now explain the last ingredient needed to answer Mather's question, which is reduction to characteristic $p$. We will use it to show that the statement of Theorem~\ref{thm:A} holds in characteristic $0$ for the power series ring $S$:

\begin{theorem}\label{thm:Achar0}
 Let $S=\bb{C}[[x_1,\cdots,x_n]]$ and consider an ideal $I=(f_1,\cdots,f_t)$ in $S$. If $f\in\ol{I}$ then $f^t\in I$.
\end{theorem}

\begin{proof} Suppose that the conclusion of the theorem fails, so $f^t\notin I$. The idea is to produce a regular ring $R$ in characteristic $p$, an ideal $J\subset R$ and an element $g\in R$ that fail the conclusion of Theorem~\ref{thm:A}, obtaining a contradiction. The point here is that the hypotheses of Theorem~\ref{thm:A} depend only on finite amount of data, which can be carried over to positive characteristic: this is essential in any argument involving reduction to characteristic $p$.
 
We now need a major theorem:  N\'eron desingularization \cite{art-rot} states that we can write $S$ as a directed union of smooth $\bb{C}[x_1,\cdots,x_n]$-algebras. This amazing
theorem allows one to descend from power series, which a priori have infinitely many coefficients, to a more finite situation. Given any finite subset, $\mc{S}$ say, of $S$, we can choose one such algebra $T$ containing $\mc{S}$. According to the equivalent description of the condition $f\in\ol{I}$ in part (b) of Theorem~\ref{thm:intclosure}, there exist $k$ and elements $s_i\in I^i$ such that
\begin{equation}\label{eq:finIbar}
f^k=\sum_i s_i\cdot f^{k-i}. 
\end{equation}
To express the containment $s_i\in I^i$, we choose coefficients $c^i_{\a}\in S$ such that 
\begin{equation}\label{eq:siinIi}
s_i=\sum_{\substack{\a=(\a_1,\cdots,\a_t)\\ \a_1+\cdots+\a_t=i}}c^i_{\a}\cdot f_1^{\a_1}\cdots f_t^{\a_t}.
\end{equation}
We will then require the smooth subalgebra $T$ of $S$ to contain $\mc{S}=\{f,f_j,s_i,c^i_{\a}\}$ for all $j,i$ and $\a$. We write $I_T$ for the ideal $(f_1,\cdots,f_t)$ of $T$. Since $f^t\notin I=I_T S$, it must be that $f^t\notin I_T$. Since $T$ contains $\mc{S}$, (\ref{eq:siinIi}) can be interpreted as an equality in $T$, which yields $s_i\in I_T^i$. Furthermore, (\ref{eq:finIbar}) is an equation in $T$, so Theorem~\ref{thm:intclosure} applies to show that $f\in\ol{I_T}$.

Since $T$ is smooth over $\bb{C}[x_1,\cdots,x_n]$, it is in particular a finite type algebra over $\bb{C}$, so it can be written as a quotient of a polynomial ring $\bb{C}[y_1,\cdots,y_r]$ by some ideal $(h_1,\cdots,h_s)$. Each of the elements of $\mc{S}$ is then represented by the class of some polynomial in $\bb{C}[y_1,\cdots,y_r]$, so collecting the coefficients of all these polynomials, as well as the coefficients of $h_1,\cdots,h_s$, we obtain a finite subset $\mc{A}\subset\bb{C}$. We define $A=\bb{Z}[\mc{A}]$ to be the smallest subring of $\bb{C}$ containing $\mc{A}$. $A$ is a finitely generated $\bb{Z}$-algebra, and we can consider the ring $T_A=A[y_1,\cdots,y_r]/(h_1,\cdots,h_s)$. $T_A$ is called a \defi{model} of $T$, having the property that $T_A\oo_A\bb{C}=T$. Moreover, $T_A$ contains all the elements of $\mc{S}$ (we use here an abuse of language: what we mean is that if we think of $A[y_1,\cdots,y_r]$ as a subring of $\bb{C}[y_1,\cdots,y_r]$, then every element of $\mc{S}$ is represented by some 
polynomial in $A[y_1,\cdots,y_r]$). We write $I_{T_A}$ for the ideal of $T_A$ generated by $f_1,\cdots,f_t$, and conclude as before that $f^t\notin I_{T_A}$ and $f\in\ol{I_{T_A}}$.

We are now ready to pass to characteristic $p>0$. We first need to observe that if we write $Q(A)$ for the quotient field of $A$, then $T_A\oo_A Q(A)$ is smooth over $Q(A)$, i.e. the map $A\to T_A$ is generically smooth: this follows from the fact that $T$ is smooth over $\bb{C}$, together with the fact that applying the Jacobian criterion to the map $\bb{C}\to T=\bb{C}[y_1,\cdots,y_r]/(h_1,\cdots,h_s)$ is the same as applying it to $Q(A)\to T_A\oo_A Q(A)=Q(A)[y_1,\cdots,y_r]/(h_1,\cdots,h_s)$. It follows that for a generic choice of a maximal ideal $\n\subset A$, the quotient $R=T_A/\n T_A$ is smooth over the finite field $A/\n$, so in particular it is a regular local ring. Writing $g$ (resp. $g_i$) for the class of $f\in T_A$ (resp. $f_i\in T_A$) in the quotient ring $R$, letting $J=(g_1,\cdots,g_t)$, and observing that the equations (\ref{eq:finIbar}) and (\ref{eq:siinIi}) descend to $R$, we get that $g\in\ol{J}$. The condition $g^t\notin J$ follows from generic flatness and the genericity assumption on 
$\n$: $T_A$ is a finite type algebra over $A$, and multiplication by $f^t$ on $(T_A/I_A)\oo_A Q(A)$ is non-zero (if it were zero, then it would also be zero on $(T_A/I_A)\oo_A Q(A)\oo_{Q(A)}\bb{C}=T/I$, but $f^t\notin I$), i.e. multiplication by $f^t$ on $T_A/I_A$ is generically non-zero. It follows that for a generic choice of $\n$, the ring $R$ is a regular local ring in characteristic $p>0$, containing an ideal $J=(g_1,\cdots,g_t)$ and an element $g\in\ol{J}$ with $g^t\notin J$. This is in contradiction with Theorem~\ref{thm:A}, concluding our proof.
\end{proof}

\subsection*{Exercises:} \ 

\bigskip

\begin{enumerate}
 \item Let $S=\bb{C}[[x_1,\cdots,x_n]]$, let $f\in S$ with $f(0) = 0$, and let $\m$ be the maximal
ideal of~$S$. Prove that $f\in \overline{\m\cdot J(f)}$. It is not known whether or not $f\in \m\cdot\overline{J(f)}$.
If true, this would give a positive solution to the Eisenbud-Mazur conjecture (see Exercise~\ref{exer:EisenbudMazur} in Section~\ref{sec:symbolic}).

\item Let $f(t), g(t)$ be polynomials with coefficients in a ring $R$, say
$f(t) = a_nt^n+\cdots+a_0$, and $g(t) = b_nt^n+\cdots+b_0$.  Let $c_i$ be the coefficient
of $t^i$ in the product $fg$. Prove that the ideal generated by $a_ib_j$ is integral
over the ideal generated by $c_{2n},\cdots,c_0$.

\item Let $S$ be a polynomial ring in $n$ variables, and let $g_1,\cdots,g_n$ be a regular
sequence of forms of degree $d$ (equivalently assume that they are forms of degree
$d$, and that the radical of the ideal they generate is the homogeneous maximal ideal).
Prove that $\ol{(g_1,\cdots,g_n)}=\m^d$.
\end{enumerate}
\medskip
\section{Uniform Artin Rees}
\bigskip

In the last section we saw how to use characteristic $p$ techniques (in a power series ring over the field $\bb{C}$) 
in order to give a uniform bound on the power of an element
in the integral closure of an ideal $I$ to be contained in $I$. A more general result was first proved
by Brian\c con and Skoda \cite{bri-sko} for convergent power
series over the complex numbers,  and later generalized to arbitrary regular
local rings by Lipman and Sathaye \cite{lip-sat}. 

\begin{theorem} \cite{bri-sko,lip-sat} Let $R$ be a regular local ring and let
$I$ be an ideal generated by $\ell$ elements. Then for all
$n\geq \ell$, $$\overline{I^n}\subseteq I^{n-\ell+1}.$$
\end{theorem}

Although apparently $\ell$ depends on the number of generators of $I$, in fact it can be made uniform. This
is because if the residue field of $R$ is infinite, every ideal is integral over an ideal generated by
$d$ elements, where $d$ is the dimension of $R$.  If the residue field is not infinite, then one can make
a flat base change to that case and still prove that one can always choose $\ell = d$. 

What about for Noetherian local rings which are not regular?  Is there a uniform integer $k$ such that
for all ideals $I$, if $n\geq k$ then $$\overline{I^n}\subseteq I^{n-k+1}?$$

The following conjecture was made in \cite{Hu}:

\begin{conjecture} Let $R$ be a reduced excellent Noetherian ring
of finite Krull dimension.  There exists an integer $k$, depending
only on $R$, such that for every ideal $I\subseteq R$, and all $n\geq k$,
$$
\overline{I^{n}}\subseteq I^{n-k}.
$$

\end{conjecture}

With a little thought, it is easy to see that at least one cannot choose such a $k$ equal to the dimension of the
ambient ring. For example if the dimension is one, then the statement $\overline{I^n}\subseteq I^{n-1+1}$
forces powers of all ideals to be integrally closed. For a local one-dimensional ring, this in turn forces the
ring to be regular. On the other hand there will often be a uniform $k$ in this special case. For example,
let's suppose that $R$ is a one-dimensional complete local domain. Its integral closure will be a DVR, say V.
The integral closure of any ideal $J$ in $R$ is given by $JV\cap R$.  There is a conductor ideal which is
primary to the maximal ideal $\m$, so there is a fixed integer $k$ such that for every ideal $I$ of $R$,
$I^k$ is in the conductor. But then,
$$\overline{I^n} = I^nV\cap R\subset I^{n-k},$$
since $I^kV\subset R$.  Thinking about this analysis, it is not totally surprising that 
this question is closely connected with another uniform question dealing with the classical
lemma of Artin and Rees.

The usual Artin-Rees lemma states that if $R$ is Noetherian, $N\subseteq M$
are finitely generated $R$-modules, and $I$ is an ideal of $R$, then there
exists a $k > 0$ (depending on $I,M,N)$ such that for all $n > k$,
$ I^{n}M\cap N =  I^{n-k}(I^kM\cap N)$.
A weaker statement which is sometimes just as useful is that for all $n > k$,
$$
I^{n}M\cap N \subseteq  I^{n-k}N.
$$

How dependent upon $I$, $M$ and $N$ is the least such $k$?  It is very easy to see that $k$ fully depends upon both $N$ and $M$, so
the only uniformity that might occur is in varying the ideal $I$. The usual proof of the Artin-Rees lemma passes to the
module $\mathfrak M: = M\oplus IM\oplus I^2M\oplus ...$, which is finitely generated over the Rees algebra $R[It]$ of $I$. Since the Rees algebra is
Noetherian, every submodule of $\mathfrak M$ is finitely generated. Applying this fact to the submodule $\mathfrak N: = N\oplus IM\cap N\oplus I^2M\cap N\oplus ....$ of
$\mathfrak M$ 
then easily gives the Artin-Rees lemma. On the face of it, there is no way that the integer $k$ could be chosen uniformly, since it depends on
the degrees of the generators of the submodule $\mathfrak N$ over the Rees algebra of $I$. Nonetheless, one can still make a 
rather optimistic conjecture \cite{Hu}:

 \begin{conjecture}Let $R$ be an excellent Noetherian ring of
finite Krull dimension.  Let $N\subseteq M$ be two finitely
generated $R$-modules. There
exists an integer $k= k(N,M)$ such that for all ideals $I\subseteq
R$ and all $n\geq k$,
$$
I^{n}M\cap N\subseteq I^{n-k}N.
$$
\end{conjecture}

In this generality, the conjecture is open. However, there is considerable literature giving lots of information about this conjecture and related problems.
See, for example, \cite{ Ab, DO, ELSV, GP, O, OP, P1, S, T}. 

It turns out that there is a very close relationship between these two conjectures, which is not at all apparent. One way to see such a connection is through
results related to tight closure theory. Suppose that $R$ is a $d$-dimensional local complete Noetherian ring of
characteristic $p$ which is reduced.  The so-called ``tight closure Brian\c con-Skoda theorem" \cite{HH1} states that
for every ideal $I$, $\overline{I^n}\subseteq (I^{n-d+1})^*,$ where $J^*$ denotes the tight closure of an ideal $J$. If $R$ is regular, every ideal
is tightly closed. The point here is that $R$ will have a non-zero test element $c$, not in any minimal prime. This means that
$c$ multiplies the tight closure of any ideal back into the ideal.  Such elements are {\it uniform} annihilators, and are one of the most important
features in the theory of tight closure. Suppose that there is a uniform Artin-Rees number $k$ for the
pair of $R$-modules, $(c)\subset R$. Then for every ideal $I$,
$$c \overline{I^n}\subseteq c(I^{n-d+1})^*\subset (c)\cap I^{n-d+1}\subset cI^{n-d-k+1}.$$
Since $c$ is not in any minimal prime and $R$ is reduced, it follows that $c$ is a non-zerodivisor. We can cancel it to obtain that
$$\overline{I^n}\subseteq I^{n-d-k+1}.$$ Thus in this case, uniform Artin-Rees implies uniform Brian\c con-Skoda. In fact these
conjectures are more or less equivalent. 

Both conjectures were proved in fairly great generality in \cite{Hu}:

\begin{theorem} (Uniform Artin-Rees) Let $S$ be a Noetherian
ring. Let $N\subseteq M$ be two finitely generated $S$-modules. If $S$
satisfies any of the conditions below,  then there
exists an integer $k$ such that for all ideals $I$ of $S$, and for all
$n\geq k$
$$
I^nM\cap N\subseteq I^{n-k}N.
$$

i) $S$ is essentially of finite type over a Noetherian local ring.

ii) $S$ is a reduced ring of characteristic $p$, 
and $S^{1/p}$ is
module-finite over $S$.

iii) $S$ is essentially of finite type over ${\bb Z}$.
\end{theorem}

\noindent We also have the following:

\begin{theorem} (Uniform Brian\c con-Skoda) Let $S$ be a Noetherian
reduced ring. If $S$ satisfies any of the following conditions, then there exists a positive integer $k$ such that for all
ideals $I$ of $S$, and for all $n\geq k$,
\[\overline{I^n}\subseteq I^{n-k}.\]

i) $S$ is essentially of finite type over an excellent Noetherian local ring.

ii) $S$ is of characteristic $p$, 
and $S^{1/p}$ is
module-finite over $S$.

iii) $S$ is essentially of finite type over ${\mathbb Z}$.
\end{theorem}

\subsection*{Exercises:} \ 

\bigskip

\begin{enumerate}

\item
If a Noetherian ring $R$ has the uniform Artin-Rees property, show that
$R/J$ (for any ideal $J \subseteq R)$ also has the uniform Artin-Rees property.
\medskip
\item
If a Noetherian ring $R$ has the uniform Artin-Rees property and $W$ is
any multiplicatively closed subset of $R$, show that the localization $R_W$ also has the
uniform Artin-Rees property.
\medskip
\item
Suppose that a Noetherian ring $R$ has the uniform Artin-Rees property.
Given a finitely generated $R$-module $M$, and an integer $i \geq 1$, 
 show that there exists an
integer $k \geq 1$ such that for all ideals $I$ of $R$ and all $n$,
$$
I^{k}\operatorname{Tor}^R_i(R/I^n,M) = 0.
$$
\medskip
\item
Let $R = k[[x,y]]$, $k$ a field.  Set $I = (x^n,y^n,x^{n-1}y)$,
$J = (x^n,y^n)$.  Prove that if $k<n$ then $I^\ell \not= J^{\ell-k}I^k$
for some $\ell \geq k+1$.
\medskip
\item Let $R$ be a Noetherian domain which satisfies the uniform Artin-Rees
theorem for every pair of finitely generated modules $N\subset M$. Let $f$ be a
non-zero element of $R$. Prove that there exists an integer $k$ such that for every
maximal ideal $\m$ of $R$, $f\notin \m^k$.

\medskip

\end{enumerate}

\section{Symbolic powers}\label{sec:symbolic}

In this section, $S$ will denote either a polynomial ring $k[x_1,\cdots,x_n]$ over some field $k$, or a regular local ring. The guiding problem will be the comparison between regular and symbolic powers of ideals in $S$. From a uniform perspective, we would like to understand whether the equality between small regular and symbolic powers guarantees the equality of all regular and symbolic powers. As we'll see, this is a very difficult question, but it gives rise to many interesting variations. We begin with a discussion of multiplicities, which will motivate the introduction of symbolic powers.

Recall from Section~\ref{sec:introduction} that if $S$ is a polynomial ring, $I$ is a homogeneous ideal, and $R=S/I$, then $\dim_k(R_i)=h_R(i)$ is a polynomial function for sufficiently large values of $i$, and the multiplicity $e(R)$ is defined by the property that
\[h_R(i)=\frac{e(R)}{(\dim R-1)!}i^{\dim R-1}+\textrm{(lower order terms)}.\]
We can define the multiplicity of a local ring $(R,\m)$ by letting $e(R)=e(\gr_{\m}(R))$, where 
\[\gr_{\m}(R)=R/\m\oplus\m/\m^2\oplus\m^2/\m^3\oplus\cdots=\bigoplus_{i\geq 0}\m^i/\m^{i+1}\]
is the \defi{associated graded ring} of $R$ with respect to $\m$. Let's look at some examples of multiplicities:

\begin{example}\label{ex:multiplicity}
\begin{enumerate}[(a)]
 \item If $S=k[x_1,\cdots,x_n]$ and $f$ is a form of degree $d$ then $e(S/(f))=d$. If $S$ is a regular local ring, then
\[e(S/(f))=\ord(f):=\max\{n:f\in\m^n\}.\]
 
 \item If $S=k[x_1,\cdots,x_n]$, and $S/I$ is Cohen-Macaulay, having a \defi{pure} resolution
\[0\lra S(-d_c)^{\b_c}\lra\cdots\lra S(-d_2)^{\b_2}\lra S(-d_1)^{\b_1}\lra S\lra S/I\lra 0,\]
where $c=\codim(I)$, then
\[e(S/I)=\frac{\prod_{j=1}^{c}d_j}{c!}.\]
 
 \item If $X\subset\bb{P}^n$ is a set consisting of $r$ points, and if we write $R_X$ for the homogeneous coordinate ring of $X$, then 
\[e(R_X)=r.\]
Note that for a projective variety $X$, $e(R_X)$ is also called the \defi{degree} of $X$.
 
 \item If $X=\bb{G}(2,n)$ is the Grassmannian of $2$-planes in $n$-space, in its Pl\"ucker embedding, then its degree is (see for example \cite{mukai})
\[e(R_X)=\frac{1}{n-1}{2n-4\choose n-2},\]
the $(n-2)$-nd \defi{Catalan number}.
\end{enumerate}
\end{example}

The following is a natural question when studying multiplicity:

\begin{question}\label{que:multflat}
 How does the multiplicity behave under flat maps?
\end{question}

We are interested in two types of flat maps:

\begin{enumerate}
 \item[I.] A local flat ring homomorphism $(R,\m)\to(R',\m')$.

 \item[II.] A localization map $R\to R_{\p}$, where $\p$ is a prime ideal.
\end{enumerate}

For flat maps of type I, the behavior of multiplicity is the subject of an old conjecture of Lech:

\begin{conjecture}[{\cite{lech}}]\label{conj:lech}
 If $(R,\m)\to(R',\m')$ is a local flat homomorphism then \[e(R)\leq e(R').\]
\end{conjecture}

It is amazing that very little progress has been made on this conjecture, although it is  easy to state, and was made about 50 years ago!
To paraphrase a famous line of Mel Hochster, it is somewhat of an insult to our field that we cannot answer this conjecture, one way or the
other.

For flat maps of type II, if $S$ is a regular local ring and $R=S/(f)$, we have (see \cite[(38.3)]{nagata-lrings}, \cite{zariski} and Exercise~\ref{exer:multisord})
\[e(R_{\p})=\max\{n:f\in\p^nS_{\p}\cap S\}\leq\max\{n:f\in\m^n\}=e(R).\]
If we denote by $\p^{(n)}$ the intersection $\p^n S_{\p}\cap S$, also called the $n$-th \defi{symbolic power} of $\p$, then the above inequality is equivalent to the containment
\begin{equation}\label{eq:p(n)inmn}
\p^{(n)}\subseteq\m^n. 
\end{equation}

For a ring $R$, we write $\Spec(S)$ (resp. $\Max(S)$) for the collection of its prime (resp. maximal) ideals. In general, if $S=k[x_1,\cdots,x_n]$ is a polynomial ring and $\p\in\Spec(S)$, then by the Nullstellensatz \cite[Thm.~4.19]{eisCA},
\[\p=\bigcap_{\substack{\m\in\Max(S)\\ \p\subset\m}}\m.\]
The symbolic powers of $\p$ can then be described (see \cite{eis-hoc}) as
\begin{equation}\label{eq:vanishordn}
\p^{(n)}=\bigcap_{\substack{\m\in\Max(S)\\ \p\subset\m}}\m^n,
\end{equation}
generalizing the inclusion (\ref{eq:p(n)inmn}). If we think of $\p$ as defining an affine variety $X$, then (\ref{eq:vanishordn}) characterizes $\p^{(n)}$ as the polynomial functions that vanish to order $n$ at the points of $X$. Symbolic powers make sense in a more general context. If $I=\p_1\cap\cdots\cap\p_r$ is an intersection of prime ideals, then
\begin{equation}\label{eq:symbpowers}
I^{(n)}=\p_1^{(n)}\cap\cdots\cap\p_r^{(n)}.
\end{equation}
For an arbitrary ideal $I$, see \cite[Def.~8.1.1]{primer-seshadri}.

One of the main questions regarding symbolic powers is the following:

\begin{question}[Regular versus symbolic powers]\label{que:powers}
 How do $I^n$ and $I^{(n)}$ compare? In particular, when are they equal for all $n$?
\end{question}

\begin{example}\label{ex:symbpowers}
\begin{enumerate}[(a)]
 \item If $I$ is a \defi{complete intersection ideal}, i.e. if it is generated by a regular sequence, then $I^n=I^{(n)}$ for all $n\geq 1$ (see Exercise~\ref{exer:ciprime} for the case when $I$ is a prime ideal).
 
 \item Let $X$ denote a generic $n\times n$ matrix with $n\geq 3$, let $\Delta=\det(X)$ and $I=I_{n-1}(X)$, the ideal of $(n-1)\times(n-1)$ minors of $X$. We have on one hand that the adjoint matrix $\adj(X)$ has entries in $I$, so $\det(\adj(X))\in I^n$, and on the other hand $\adj(X)\cdot X=\Delta\cdot\bb{I}_n$ (where $\bb{I}_n$ denotes the $n\times n$ identity matrix), so $\det(\adj(X))=\Delta^{n-1}$. We get $\Delta^{n-1}\in I^n$, from which it can be shown that $\Delta\in I^{(2)}\setminus I^2$ (see Exercise~\ref{exer:symbvsregpower}).

 \item If $S=k[x,y,z]$, $\p\in\Spec(S)$ with $\dim(S/\p)=1$, then the following are equivalent:
 \begin{itemize}
  \item $\p^{(n)}=\p^n$ for all $n\geq 1$.
  \item $\p^{(2)}=\p^2$.
  \item $\p$ is locally a complete intersection.
 \end{itemize}
\end{enumerate}
\end{example}

The following conjecture of H\"ubl is related to Example~\ref{ex:symbpowers}(b), and in particular it is true in the said example:
\begin{conjecture}[{\cite[Conj.~1.3]{hubl}}]\label{conj:hubl}
 If $R$ is a regular local ring, $\p\in\Spec(R)$ and $f\in R$, with the property that $f^{n-1}\in\p^n$, then $f\in\m\cdot\p$.
\end{conjecture}
\noindent As a consequence of Exercise~\ref{exer:symbvsregpower}, it follows under the assumptions of the conjecture that $f\in\p^{(2)}$, but the conclusion $f\in\m\cdot\p$ turns out to be significantly harder. In characteristic $0$, Conjecture~\ref{conj:hubl} is equivalent to the Eisenbud-Mazur conjecture on evolutions (see Exercise~\ref{exer:EisenbudMazur} and \cites{eis-maz,boocher}).

A natural uniformity problem is to determine if it is enough to test the equality in Question~\ref{que:powers} for finitely many values of $n$, and moreover to determine a uniform bound for these values. A precise version of this is the following:

\begin{question}\label{que:symbuniform}
 Assume that $S$ is a regular local ring, or a polynomial ring, and that $\p\in\Spec(S)$. If $\p^{(\dim S)}=\p^{\dim S}$, does it follow that $\p^{(n)}=\p^n$ for all $n\geq 1$?
\end{question}

One could ask the same question, replacing the condition $\p^{(\dim S)}=\p^{\dim S}$ with a stronger one, namely $\p^{(i)}=\p^{i}$ for all $i\leq\dim(S)$. The equivalence between the two formulations is unknown in general, but in characteristic zero it would follow from a positive answer to the following:

\begin{question}
 Let $S$ be a regular local ring containing $\bb{C}$, and let $\p\in\Spec(S)$. Does it follow that there exists a non-zerodivisor of degree $1$ in the associated graded ring $\gr_{\p}(S)$?
\end{question}

There are two test cases where much is known about Question~\ref{que:symbuniform}:
\begin{enumerate}
 \item[I.] Points in $\bb{P}^2$.

 \item[II.] Square-free monomial ideals.
\end{enumerate}

\ul{Case I: points in $\bb{P}_k^2$.} Consider a set $X$ of $r$ points in $\bb{P}_k^2$, and let $I=I_X$ be its defining ideal. Since $\codim(I)=2$, it follows from \cites{ein-laz-smi,hochster-huneke} that $I^{(2n)}\subset I^n$ for all $n$, so in particular $I^{(4)}\subset I^2$. It is then natural to ask

\begin{question}[{\cite[Que.~0.4]{huneke-open}}]\label{que:I3inI2}
 Is it true that $I^{(3)}\subset I^2$?
\end{question}

Bocci and Harbourne gave a positive answer to this question when $I=I_X$ is the ideal of a generic set of points \cite{bocci-harbourne}. In characteristic $2$, the inclusion follows from the techniques of \cite{hochster-huneke} (see \cite[Example~8.4.4]{primer-seshadri}). Harbourne formulated some general conjectures for arbitrary homogeneous ideals which would imply a positive answer to Question~\ref{que:I3inI2} (see \cite[Conj.~8.4.2,\ Conj.~8.4.3]{primer-seshadri} or \cite[Conj.~4.1.1]{harbourne-huneke}). Unfortunately, it turns out that the relation between symbolic powers and ordinary powers is much more subtle, even in the case of points in $\bb{P}^2$: Question~\ref{que:I3inI2} was recently given a negative answer in \cite{dum-sze-gas}.

One measure of how close the ordinary powers are to symbolic powers is given by comparing the least degrees of their generators. Given a homogeneous ideal $J$, we write $\a(J)$ for the smallest degree of a minimal generator of $J$. Since $I^n\subset I^{(n)}$, $\a(I^{(n)})\leq \a(I^n)=n\cdot\a(I)$, or equivalently $\a(I)\geq\a(I^{(n)})/n$. The sequence $\a(I^{(n)})/n$ is always convergent (see Exercise~\ref{exer:limalphaI}), and it has made a surprising appearance in the construction of Nagata's counter-example to Hilbert's 14th problem.

\begin{theorem}[{\cite{nagata-14th}}]
 If $X$ is a set of $r$ generic points in $\bb{P}^2_{\bb{C}}$ and $I=I_X$ is its defining ideal, then
\begin{enumerate}
 \item $\ds\lim_{n\to\infty}\ds\frac{\a(I^{(n)})}{n}\leq\sqrt{r}$.
 \item If $r=s^2$ is a perfect square with $s\geq 4$, then $\ds\frac{\a(I^{(n)})}{n}>s$.
\end{enumerate}
If we take $r=s^2$ with $s\geq 4$, then the ring $R=\bigoplus_{n\geq 0}I^{(n)}$ is a ring of invariants that is not finitely generated.
\end{theorem}

\ul{Case II: square-free monomial ideals.} Consider the ideal 
\[I=(xy,xz,yz)=(x,y)\cap(x,z)\cap(y,z)\]
defining a set of $3$ non-collinear points in $\bb{P}^2$. We have 
\[I^{(2)}=(x,y)^{2}\cap(x,z)^{2}\cap(y,z)^{2},\]
and it is easily checked that $xyz\in I^{(2)}\setminus I^2$.

More generally, any square-free monomial ideal $I$ can be written as an intersection $\p_1\cap\cdots\cap\p_s$ of prime ideals, where each $\p_i$ is generated by a subset of the variables. If $\codim(\p_i)=c_i$ then $x_1\cdots x_n\in\p_i^{c_i}$. Taking $c=\codim(I)=\min(c_i)$ we get that
\[x_1\cdots x_n\in\p_1^c\cap\cdots\cap\p_s^c=I^{(c)}.\]
It follows that if $I^c=I^{(c)}$, then $x_1\cdots x_n$ must be contained in $I^c$; thus $I$ contains $c$ monomials with disjoint support, i.e.
\begin{equation}\label{eq:regseqmonomials}
I\textrm{ contains a regular sequence consisting of }c\textrm{ monomials.}
\end{equation}

If $I^{(n)}=I^n$ for all $n$ then (\ref{eq:regseqmonomials}) holds for all ideals $J$ obtained from $I$ by setting variables equal to $0$ or $1$ (such an ideal $J$ is called a \defi{minor} of $I$). This raises the following question.

\begin{question}[Gitler-Valencia-Villarreal {\cite{GVV}}]\label{que:GVV}
 If (\ref{eq:regseqmonomials}) holds for all minors of $I$, does it follow that $I^{(n)}=I^n$ for all $n$?
\end{question}
This question is equivalent to a Max-Flow-Min-Cut conjecture due to Conforti and Cornu{\'e}jols \cite[Conj.~1.6]{corn}, and it is open except in the case when $I$ is generated by quadrics. Note that to any square-free monomial ideal $I$ generated by quadrics one can associate a graph $G$ as follows: the vertices of $G$ correspond to the variables in the ring, and two vertices $x_i$ and $x_j$ are joined by an edge if $x_ix_j\in I$. Conversely, starting with a graph $G$ one can reverse the preceding construction to get a monomial ideal $I$ generated by quadrics. $I$ is called the \defi{edge ideal} of the graph $G$. With this terminology, we have the following:

\begin{theorem}[{\cite{GVV}}]\label{thm:edgeideal}
 If $I$ is the edge ideal of a graph $G$, then the following are equivalent:
\begin{enumerate}
 \item $I^{(n)}=I^n$ for all $n\geq 1$.
 \item $I$ is \defi{packed}, i.e. (\ref{eq:regseqmonomials}) holds after setting any subset of the variables to be equal to $0$ or $1$.
 \item $G$ is bipartite.
\end{enumerate}
\end{theorem}

\subsection*{Exercises:} \ 

\bigskip

\begin{enumerate}

\item A famous theorem of Rees says that if $R$ is a Noetherian local ring which is
formally equidimensional (i.e., its completion is equidimensional), and $I$ is primary to
the maximal ideal $\m$, then $f\in\ol{I}$ if and only if $e(I)=e(I+(f))$. Prove the easy
direction of this theorem.

\item Let $S=k[[x_1,\cdots,x_n]]$ and let $f_i = x_1^{a_{i1}}+\cdots+x_n^{a_{in}}$. Assume that the ideal $I$ generated by
the $f_i$'s has the property that $S/I$ is a finite dimensional vector space. Give a formula, in terms of the exponents $a_{ij}$, 
for the dimension of this vector space (which is also the length of $S/I$, or the multiplicity of the ideal $I$).

\item\label{exer:ciprime} Let $\p$ be a prime ideal generated by a regular sequence in a regular local 
ring (or polynomial ring). Prove that $\p^{(n)} = \p^n$ for all $n\geq 1$.

\item Prove that if $I$ is a reduced ideal in a polynomial ring, then $I^{(n)}\cdot I^{(m)}\subset I^{(n+m)}$.

\item With the notation from the previous exercise, prove that the graded algebra $$T: = \bigoplus_{n\geq 0} I^{(n)}$$ is Noetherian if and only if
there exists an integer $k$ such that for all $n$, $$(I^{(k)})^n = I^{(kn)}.$$

\item\label{exer:multisord} Let $S$ be a regular local ring and let $f\in S$ be a nonzero, nonunit element in $S$.
Prove that the multiplicity of $R: = S/(f)$ is equal to the order of $f$.

\item Prove the following result of Chudnovsky \cite{Ch}, which was  proved by him using transcendental
methods: if $I$ is the ideal of a set of points in the projective plane over the complex numbers,
then $$ \alpha(I^{(N)})\geq \frac{N\alpha(I)}{2},$$
where (as before) $\alpha(\quad)$ denotes the least degree of a minimal generator of a homogeneous ideal.

\item Let $I$ be the ideal of at most five points in the projective plane. Prove that
$I^{(3)}\subset I^2$.

\item Let $I$ be an ideal of  points in the projective plane over a field of
characteristic $2$. Prove that
$I^{(3)}\subset I^2$.

\item\label{exer:EisenbudMazur} The Eisenbud-Mazur conjecture states that if $S$ is a power series ring over a
field of characteristic $0$, then for every prime ideal $\p$, $$\p^{(2)}\subset \m \p,$$ where $\m$ denotes the maximal ideal of $S$. 
Prove this when $\p$ is homogeneous.

\item Prove the Eisenbud-Mazur conjecture assuming that for every  $f\in S$ ($S$ as in Exercise~\ref{exer:EisenbudMazur} and
$f$ not a unit) $f$ is not a minimal generator of the integral closure of its partial derivatives.

\item\label{exer:symbvsregpower} Let $R$ be a regular local ring, and let $\p$ be a prime ideal. Set $G = \gr_{\p}(R)$,
the associated graded ring of $\p$. If $f\in R$, write $f^*$ for the leading form of $f$ in $G$.
Show that if $f^*$ is nilpotent in $G$, then $f\in \p^{(n)}$ but $f\notin \p^n$ for some $n$.

\item\label{exer:limalphaI} Let $I$ be a homogeneous ideal in a polynomial ring $S$, satisfying $I=\sqrt{I}$. Prove that
the limit of $\ds\frac{\alpha(I^{(m)})}{m}$ exists as $m$ goes to infinity.
\end{enumerate}

\end{document}